\documentclass[a4paper,11pt]{amsart}
\usepackage{amsmath,amssymb}
\usepackage[english]{babel}
\usepackage[left=2.5cm,right=2.5cm,top=4cm,bottom=2.5cm]{geometry}
\usepackage{tikz}

\newtheorem{theorem}{Theorem}[section]
\newtheorem{lemma}[theorem]{Lemma}

\newtheorem{remark}[theorem]{Remark}
\newtheorem{proposition}[theorem]{Proposition}
\def\dfrac{\displaystyle\frac}
\def\dsum{\displaystyle\sum}
\def\dprod{\displaystyle\prod}
\begin{document}

\title[Catenary System Identification]{\textbf{\vspace{-3mm} \\}IDENTIFICATION OF A REVERSIBLE catenary SYSTEM BY THE MEASURE OF ONE COMPARTMENT}
\author[B. FADIS]{B. FADIS$^{1}$}
\address{$^{1}$Higher School of Applied Sciences of Algiers, Algeria.} 
\email{$^{1}$badreddine.f@gmail.com}
\author[B. HEBRI]{B. HEBRI$^{2}$}
\author[S.KHELIFA]{S.KHELIFA$^{3}$}
\address{$^{2,3}$University of Sciences and Technology Houari Boumediene (USTHB),\newline
\indent Faculty of Mathematics, Department of Calculus\vspace{1mm} \\P.B. 32 El-Alia, 16111, Bab Ezzouar, Algiers,\newline 
\indent Algeria.}
\email{$^{2}$bhebri@usthb.dz,\ $^{3}$ksices@yahoo.fr}

\begin{abstract}
 The known results on compartment system's identification requires observations and measures on some compartments. However, in many situations some compartments can not be measured, so the minimization of the compartment's number to be measured is a very important problem to solve. In this work, this problem had been solved for a class of compartmental systems. It is proved that all reversible catenary system is identifiable by the measure of one compartment : the principal compartment. \vspace{0.2cm}\\
\noindent\textsc{Keywords and phrases.} Differential systems; Compartmental analysis; Catenary systems; Will problems; Parameters identification.
\end{abstract}

\maketitle

\section{Introduction}
Catenary systems represents an important class in compartment systems.They are in particular used for modeling transport process through successive biological structures like, for example, the diffusion of some drugs in
human body when they are injected and diffused in the blood and finally eliminated by urinary tract. In homogenous media, modeling by using catenary systems can be very convenient to give a discrete representation
of some process. catenary systems can constitute a simplification of a complex process to estimate cellular kinetic parameters (see Golgi [6]).\vspace{0.2cm}\\
Catenary systems are also used in ecology for the representation of an age structured population. Some equivalents models had been proposed in epidemiology for the Bilharzia (see Cohen \& Rochet [5]). All known
identification's methods for this type of models are essentially based on empiric algorithms requiring measurements on several compartments see [1],[2],[3],[4], to the different posed problems in compartmental analysis
and their best comprehension.\vspace{1cm}

\section{Preliminary results}

\subsection{Some generalities on catenary systems}\textcolor[rgb]{1.00,1.00,1.00}{aaaa}\vspace{0.2cm}\\
Let $n$ be an integer greater than $3$. A catenary system is a
compartmental system which can be drafted by :
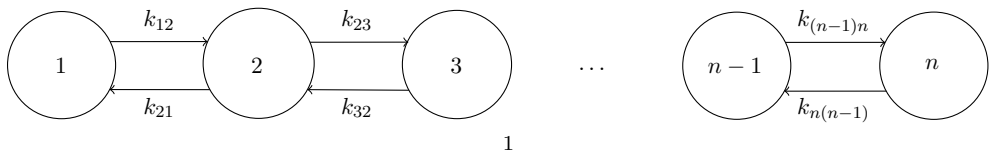
\begin{figure}[h]
\begin{tikzpicture}[transform shape, scale = 0.8]
\clip(4.189032095827969,2.591002257939712) rectangle (20.979296498549257,5.601837812440735);
\draw(5.24,4.1) circle (0.8883692925805119cm);
\draw(8.48,4.13) circle (0.9128094311091632cm);
\draw(11.74,4.11) circle (0.8975231574396313cm);
\draw(16.34,4.09) circle (0.8883692925805122cm);
\draw(19.58,4.09) circle (0.8892099407648747cm);
\draw [->] (6.038438288250357,4.4894820918294425) -- (7.641414237713173,4.490548161570204);
\draw [->] (7.678372418819462,3.6934060483555986) -- (6.030687194405471,3.695013876036222);
\draw [->] (9.323317256286447,4.479338035105443) -- (10.940465799061892,4.480975538071364);
\draw [->] (10.95018887918184,3.683682968235651) -- (9.28168683453003,3.693514860661836);
\draw [->] (17.1397317152207,4.476819316568023) -- (18.77958825826795,4.477343984660325);
\draw [->] (18.796714535979596,3.6690983480557295) -- (17.121852913796356,3.668210928083544);
\draw (6.45,5.15) node[anchor=north west] {$k_{12}$};
\draw (6.45,3.7) node[anchor=north west] {$k_{21}$};
\draw (9.7,5.15) node[anchor=north west] {$k_{23}$};
\draw (9.7,3.7) node[anchor=north west] {$k_{32}$};
\draw (17.20,5.15) node[anchor=north west] {$k_{(n-1)n}$};
\draw (17.20,3.7) node[anchor=north west] {$k_{n(n-1)}$};
\draw (5.,4.35) node[anchor=north west] {1};
\draw (8.22,4.35) node[anchor=north west] {2};
\draw (11.5,4.35) node[anchor=north west] {3};
\draw (15.73,4.35) node[anchor=north west] {$n-1$};
\draw (19.3,4.33) node[anchor=north west] {$n$};
\draw (13.6,4.2) node[anchor=north west] {$\ldots$};
\end{tikzpicture}\vspace{-0.5cm}
\end{figure}

The exchange coefficients $k_{ij}$ for this kind of systems are characterized by:
\[
\left\{
\begin{array}{l}
k_{ij}=0,\ \forall (i,j)\in \left\{ 1,2,\ldots ,n\right\} ^{2},\ \left\vert i-j\right\vert \geq 2, \vspace{0.2cm}\\
k_{i(i+1)}\neq 0\;;\;\forall i\in \left\{ 1,2,\ldots ,n-1\right\} .%
\end{array}%
\right.\vspace{0.4cm}
\]
Catenary systems can be classified into two classes:
\begin{itemize}
\item the reversible systems caracterized by:
\[
k_{(j+1)j}\neq 0,\ \forall j\in \left\{ 1,2,\ldots ,n-1\right\},
\]
\item the irreversible systems caracterized by:
\[
\exists j_{0}\in \left\{ 1,2,\ldots ,n-1\right\} \;\text{such that }%
k_{(j_{0}+1)j_{0}}=0.\vspace{0.3cm}
\]
\end{itemize}
These systems are governed by the following differential system noted (S1), given by:
\begin{equation}
\left\{
\begin{array}{l}
x_{1}^{\prime }(t)=k_{11}x_{1}(t)+k_{21}x_{2}(t),\vspace{0.2cm} \\
x_{j}^{\prime}(t)=k_{(j-1)j}x_{j-1}(t)+k_{jj}x_{j}(t)+k_{(j+1)j}x_{j+1}(t),\ \forall j\in \left\{ 2,3,\ldots ,n-1\right\}, \vspace{0.2cm} \\
x_{n}^{\prime }(t)=k_{(n-1)n}x_{n-1}(t)+k_{nn}x_{n}(t),
\end{array}%
\right.  \tag{S1}  \label{0}\vspace{0.2cm}
\end{equation}%
where
\[
\left\{
\begin{array}{l}
k_{11}=-(k_{12}+k_{1e}), \vspace{0.2cm}\\
k_{jj}=-(k_{j(j-1)}+k_{j(j+1)})\;,\;\;\forall j\in \left\{ 2,3,\ldots
,n-1\right\} , \vspace{0.2cm}\\
k_{nn}=-k_{n(n-1)}.%
\end{array}%
\right.\vspace{0.2cm}
\]

We associate the matrix $A$, $A=\left( k_{ij}\right) _{1\leq i,j\leq n}$ to this open $n$-compartmental system. We can remark that $A$ is a tridiagonal and compartmental matrix so its eigenvalues noted $\lambda _{i}$ for $i\in
\left\{ 1,2,\ldots ,n\right\} $ are reals, strictly negatives and different (see [8]). The general solution of the system (S1) can be written under the
form:
\[
x_{j}(t)=\dsum\limits_{i=1}^{n}\beta _{i}^{j}\exp (\lambda_{i}t),\ \forall j\in \left\{ 1,2,\ldots ,n\right\},
\]
where $\left( \beta _{i}^{j}\right) _{1\leq i\leq n}$ is the column $j$ of
the matrix $B$ of elementaries masses associated to the compartment $j$ (see [8]).\vspace{0.2cm}\\
Measuring the primary compartment (the catalyser of the reaction) we can
determine the coefficients $\beta _{1}^{1},\beta _{2}^{1},\ldots ,\beta_{n}^{1},\lambda _{1},\lambda _{2},\ldots,\lambda _{n}$, noted by $\beta _{1}^{1*},\beta_{2}^{1*},\ldots,\beta_{n}^{1*},\lambda _{1}^{*},\lambda _{2}^{*},\ldots,\lambda _{n}^{*}$ after
minimization of the functional $J$ introduced by Y. Cherruault (see [3]):
\[
J(\beta _{1}^{1},\;\beta _{2}^{1},\ldots,\beta _{n}^{1}\ ,\lambda_{1},\lambda_{2},\ldots ,\lambda _{n})=\dsum\limits_{j=1}^{m}\dsum\limits_{i=1}^{n}\left( \beta _{i}^{1}\exp (\lambda _{i}.t_{j})-\overline{x}_{1}(t_{j})\right)^{2},
\]
where $\overline{x}_{1}(t_{j})$ is the measure of the primary compartment at time $t_{j}$ for $j\ $in$\ \left\{ 1,2,3,\ldots ,m\right\} $. These coefficients will be noted : $\beta _{1}^{1\ast },\beta _{2}^{1\ast },\ldots
,\beta _{n}^{1\ast },$ $\lambda _{1}^{\ast },\lambda _{2}^{\ast },\ldots,\lambda _{n}^{\ast}$.\vspace{0.2cm}\\
Thus, the general solution of the system (S1)\ can be writen:
\[
\begin{tabular}{c}
$x_{1}(t)=\dsum\limits_{i=1}^{n}\beta _{i}^{1\ast }\exp (\lambda _{i}^{\ast}t),$ \\
$x_{j}(t)=\dsum\limits_{i=1}^{n}\beta _{i}^{j}\exp (\lambda _{i}^{\ast}t)\;,\ \forall j\in \left\{ 2,3,\ldots,n\right\}$,
\end{tabular}%
\]
where the coefficients $\beta _{i}^{j}$ for $j\in \left\{ 2,3,\ldots
,n\right\} $\ and $i\in \left\{ 1,2,\ldots ,n\right\} $\ remain unknowns.%
\vspace{0.35cm}\newline
In order to identify the all compartmental system, an injection of substance is made in the primary compartment. So the system (S1) is completed by the initial conditions:
\[
x_{1}(0)=a\;\;\;\text{and}\;\;\;x_{j}(0)=0\;,\;\;\forall j\in \left\{
2,3,\ldots ,n\right\}.
\]
Hence
\begin{equation}
\left\{
\begin{array}{l}
\dsum\limits_{i=1}^{n}\beta _{i}^{1\ast }=x_{1}(0)=a,\vspace{0.2cm} \\
\dsum\limits_{i=1}^{n}\beta _{i}^{j}=x_{j}(0)=0,\ \forall j\in \left\{2,3,\ldots,n\right\} \;.%
\end{array}%
\right.  \tag{2.1}  \label{1}
\end{equation}%
\smallskip\ The linear independency of the functions $\varphi _{i}$, $i\in\left\{ 1,2,\ldots ,n\right\} $ defined by:
\[
\varphi _{i}(t)=\exp (\lambda _{i}^{\ast}t),\ \forall t\geqslant0,\ \forall i\in \left\{ 1,2,\ldots ,n\right\},
\]
gives, for all integer $i$ in $\left\{ 1,2,\ldots ,n\right\} $:\vspace{0.2cm}
\begin{equation}
\left\{
\begin{array}{l}
\lambda _{i}^{\ast }\beta _{i}^{1\ast }=k_{11}\beta _{i}^{1\ast}+k_{21}\beta _{i}^{2}\;,\vspace{0.2cm} \\
\lambda _{i}^{\ast }\beta _{i}^{j}=k_{(j-1)j}\beta _{i}^{j-1}+k_{jj}\beta_{i}^{j}+k_{(j+1)j}\beta _{i}^{j+1},\ \forall j\in \left\{ 2,3,\ldots,n-1\right\},\vspace{0.2cm} \\
\lambda _{i}^{\ast }\beta _{i}^{n}=k_{(n-1)n}\beta _{i}^{n-1}+k_{nn}\beta_{i}^{n}.\vspace{0.2cm}%
\end{array}%
\right.  \tag{2.2}
\end{equation}%
In the following we need to introduce the following notation, for integer $j$ in $\left\{ 1,2,\ldots ,n\right\} $ and integer $l$ in $\left\{ 0,1,\ldots ,n\right\} $, let $\Delta _{j}^{l}$ be:
\[
\Delta _{j}^{l}=\dsum\limits_{i=1}^{n}\left( \lambda _{i}^{\ast }\right)
^{l}\beta _{i}^{j}\;.
\]
\subsection{Preliminary results}\textcolor[rgb]{1.00,1.00,1.00}{aaa}\vspace{0.2cm}\\
With the previous notations, we have the following results:
\begin{lemma}\label{lem1}
The following relationships hold:
\begin{itemize}
  \item [$i)$]
  \begin{equation}
\Delta _{1}^{0}=a\;\;\;\text{and\ \ \ }\Delta _{j}^{0}=0,\ \forall j\in\left\{ 2,3,\ldots ,n\right\}.\vspace{0.2cm}  \tag{2.3}
\end{equation}%
  \item  [$ii)$]
  \begin{equation}
\Delta _{1}^{1}=k_{11}a.\vspace{0.2cm}  \tag{2.4}
\end{equation}%
  \item [$iii)$] $\forall j\in \left\{ 2,3,\ldots ,n-1\right\}$,
\begin{equation}
\Delta _{j}^{l}=k_{(j-1)j}\Delta _{j-1}^{l-1}+k_{jj}\Delta_{j}^{l-1}+k_{(j+1)j}\Delta _{j+1}^{l-1},\ \forall l\in \left\{1,2,\ldots ,j\right\}.  \tag{2.5}
\end{equation}
\item [$iv)$]
\begin{equation}
\Delta _{n}^{l}=k_{(n-1)n}\Delta _{n-1}^{l-1}+k_{nn}\Delta_{n}^{l-1},\ \forall l\in \left\{ 1,2,\ldots ,n\right\}.
\tag{2.6}
\end{equation}
\end{itemize}
\end{lemma}
\begin{proof}\textcolor[rgb]{1.00,1.00,1.00}{aaa}\vspace{-0.8cm}\\
\begin{itemize}
  \item [$i)$]
  $\Delta _{j}^{0}$\ represent initial conditions then the results hold.
  \item  [$ii)$] We recall that:
\[
\Delta _{1}^{1}=\dsum\limits_{i=1}^{n}\left( \lambda _{i}^{\ast }\right)^{1}\beta _{i}^{1}=\dsum\limits_{i=1}^{n}\lambda _{i}^{\ast }\beta_{i}^{1\ast }.
\]%
Then due to the relationship (2.2) we can write that:
\[
\Delta _{1}^{1}=\dsum\limits_{i=1}^{n}\left( k_{11}\beta _{i}^{1\ast}+k_{21}\beta _{i}^{2}\right).\]%
Or
\[
\Delta _{1}^{1}=k_{11}\dsum\limits_{i=1}^{n}\beta _{i}^{1\ast}+k_{21}\dsum\limits_{i=1}^{n}\beta _{i}^{2}.
\]%
Consequently : $$\Delta _{1}^{1}=k_{11}a.$$
  \item [$iii)$]
  Let $j$ and $l$ be two integers such that $j\in \left\{ 2,3,\ldots,n-1\right\} $ and $l\in \left\{ 1,2,\ldots ,j\right\} $, we have:
\[
\Delta _{j}^{l}=\dsum\limits_{i=1}^{n}\left( \lambda _{i}^{\ast }\right)^{l}\beta _{i}^{j}=\dsum\limits_{i=1}^{n}\left( \lambda _{i}^{\ast }\right)^{l-1}\lambda _{i}^{\ast }\beta _{i}^{j},\]%
then, the relationship (2.2) gives:
\[
\Delta _{j}^{l}=\dsum\limits_{i=1}^{n}\left( \lambda _{i}^{\ast }\right)^{l-1}\left( k_{(j-1)j}\beta _{i}^{j-1}+k_{jj}\beta_{i}^{j}+k_{(j+1)j}\beta _{i}^{j+1}\right),\]%
or
\[
\Delta _{j}^{l}=k_{(j-1)j}\dsum\limits_{i=1}^{n}\left( \lambda _{i}^{\ast}\right)^{l-1}\beta _{i}^{j-1}+k_{jj}\dsum\limits_{i=1}^{n}\left( \lambda_{i}^{\ast}\right)^{l-1}\beta _{i}^{j}+k_{(j+1)j}\dsum\limits_{i=1}^{n}\left( \lambda _{i}^{\ast }\right)^{l-1}\beta _{i}^{j+1}.
\]%
The definition of $\Delta _{q}^{p}$ let us conclude that:
\[
\Delta _{j}^{l}=k_{(j-1)j}\Delta _{j-1}^{l-1}+k_{jj}\Delta_{j}^{l-1}+k_{(j+1)j}\Delta _{j+1}^{l-1}.\vspace{0.2cm}
\]%
\item[$iv)$]
Let $l$ be an integer in $\left\{ 1,2,\ldots ,n\right\} ,$ we have:
\[
\Delta _{n}^{l}=\dsum\limits_{i=1}^{n}\left( \lambda _{i}^{\ast }\right)^{l}\beta _{i}^{n}=\dsum\limits_{i=1}^{n}\left( \lambda _{i}^{\ast }\right)^{l-1}\lambda _{i}^{\ast }\beta _{i}^{n}.
\]%
Then, the relationship (2.2) gives:
\[
\Delta _{n}^{l}=\dsum\limits_{i=1}^{n}\left( \lambda _{i}^{\ast }\right)^{l-1}\left( k_{(n-1)n}\beta _{i}^{n-1}+k_{nn}\beta _{i}^{n}\right),
\]%
or
\[
\Delta _{n}^{l}=k_{(n-1)n}\dsum\limits_{i=1}^{n}\left( \lambda _{i}^{\ast}\right) ^{l-1}\beta _{i}^{n-1}+k_{nn}\dsum\limits_{i=1}^{n}\left( \lambda_{i}^{\ast }\right) ^{l-1}\beta _{i}^{n}.
\]%
The definition of $\Delta _{q}^{p}$ let us conclude that:
\[
\Delta _{n}^{l}=k_{(n-1)n}\Delta _{n-1}^{l-1}+k_{nn}\Delta _{n}^{l-1}.
\]
\end{itemize}
\end{proof}

\begin{lemma}\label{lem2} Let $n$ be an integer strictly greater than $4$. Then:
\[
\Delta _{j}^{l}=0,\ \forall j\in \left\{ 3,4,\ldots ,n\right\},\ \forall l\in \left\{ 1,2,\ldots ,n-2\right\},\ j-l\geqslant 2.%
\vspace{0.35cm}
\]
\end{lemma}
\begin{proof}
The proof of this lemma is based on a recurrence process.\newline
$i)$ In the first step we set $l=1$:

\textbf{a-} let $j$ be an integer in $\left\{ 3,4,\ldots ,n-1\right\} .$
Owing to the relationship (2.5), we have:
\[
\Delta _{j}^{1}=k_{(j-1)j}\Delta _{j-1}^{0}+k_{jj}\Delta_{j}^{0}+k_{(j+1)j}\Delta _{j+1}^{0},
\]%
the relationships (2.3) let us conclude that:
\[
\Delta _{j}^{1}=0.
\]

\textbf{b-} when $j=n$, the relationship (2.6) gives:
\[
\Delta _{n}^{1}=k_{(n-1)n}\Delta _{n-1}^{0}+k_{nn}\Delta _{n}^{0},
\]%
so the relationships(2.3) let us conclude that:
\[
\Delta _{n}^{1}=0.\vspace{0.25cm}
\]%
$ii)$ In the last step we assume that for all integer $l$ in $\left\{2,\ldots ,n-3\right\} $ and for all integer $j$ in $\left\{ l+2,\ldots,n\right\} $ we have :
\[
\Delta _{j}^{l}=0,
\]

\textbf{a-} for any integer $j$ in $\left\{ l+3,\ldots ,n-1\right\} ,$ owing to the relationship (2.5), we have:
\[
\Delta _{j}^{l+1}=k_{(j-1)j}\Delta _{j-1}^{l}+k_{jj}\Delta_{j}^{l}+k_{(j+1)j}\Delta _{j+1}^{l},
\]%
then the recurrence hypothesis permits to affirm that:
\[
\Delta _{j-1}^{l}=\Delta _{j}^{l}=\Delta _{j+1}^{l}=0,
\]%
because $j-1-l\geqslant 2$ , $j-l\geqslant 3$\ and $j+1-l\geqslant 4$, so we can conclude that:
\[
\Delta _{j}^{l+1}=0.
\]

\textbf{b-} for $j=n$, the same process\ gives:
\[
\Delta _{n}^{l+1}=0.
\]
\end{proof}
\begin{lemma}\label{lem3} For all integer $l$ in $\left\{ 1,2,\ldots,n-1\right\} $ the following relationship holds:
\[
\Delta _{l+1}^{l}=a\left( \dprod\limits_{i=1}^{l}k_{i(i+1)}\right).
\]%
\end{lemma}
\begin{proof}
To show this result we use a recurrence process like in the previous lemma.%
\vspace{0.25cm} \newline
\textbf{1}${^{\circ }}$ For $l=1,$ we have:
\[
\Delta _{2}^{1}=\dsum\limits_{i=1}^{n}\lambda _{i}^{\ast }\beta _{i}^{2},
\]%
due to the relationship (2.5), we have:
\[
\Delta _{2}^{1}=k_{12}\Delta _{1}^{0}+k_{22}\Delta _{2}^{0}+k_{32}\Delta_{3}^{0}.
\]%
Using relationship (2.3) we can conclude:
\[
\Delta _{2}^{1}=a.k_{12}+0.k_{22}+0.k_{32}=a.k_{12}.\vspace{0.25cm}
\]%
\textbf{2}${^{\circ }}$ Assume that for all integer $l$ in $\left\{ 2,\ldots ,(p-1)\right\} $ with $p\leqslant (n-2)$ we have:
\[
\Delta _{l+1}^{l}=a\left( \dprod\limits_{i=1}^{l}k_{i(i+1)}\right).
\]%
By definition, we know that:
\[
\Delta _{p+1}^{p}=\dsum\limits_{i=1}^{n}\left( \lambda _{i}^{\ast }\right)
^{p}\beta _{i}^{p+1},
\]%
relationship (2.5), gives:
\[
\Delta _{p+1}^{p}=k_{p(p+1)}\Delta _{p}^{p-1}+k_{(p+1)(p+1)}\Delta_{p+1}^{p-1}+k_{(p+2)(p+1)}\Delta _{p+2}^{p-1}.
\]%
As Lemma \ref{lem2}:
\[
\Delta _{p+1}^{p-1}=\Delta _{p+2}^{p-1}=0,
\]%
so we can conclude that:
\[
\Delta _{p+1}^{p}=a\left( \dprod\limits_{i=1}^{p}k_{i(i+1)}\right).%
\vspace{0.25cm}
\]%
\textbf{3}${^{\circ }}$ When $l=n-1,$ using relationship (2.6), Lemma \ref{lem2} and the recurrence hypothesis, we obtain:
\[
\Delta _{n}^{n-1}=a.k_{(n-1)n}.\left(\dprod\limits_{i=1}^{n-2}k_{i(i+1)}\right)+0.k_{nn},
\]%
or else:
\[
\Delta _{n}^{n-1}=a\left( \dprod\limits_{i=1}^{n-1}k_{i(i+1)}\right).%
\vspace{0.5cm}
\]\end{proof}
\begin{remark}\label{rem1}
Knowing that $a\neq 0,\ $from Lemma \ref{lem3} it becomes:
\[
\Delta _{l+1}^{l}\neq 0,\ \forall l\in \left\{ 1,2,\ldots ,n-1\right\}.
\]%
\end{remark}

\section{Main result}

\subsection{General relationships}\textcolor[rgb]{1.00,1.00,1.00}{aaa}\vspace{0.2cm}\\
The main result can be obtained as a consequence of the two following propositions. The compartmental matrix $A$ and the matrix $B$ of elementaries masses (see [8]) will be alternately identified, respectively row after
row and column after column.
\begin{proposition}\label{prop1} If the catenary system (S1) is reversible, then the following relationships hold:\vspace{0.2cm}\\
$i)$\textit{\ }%
\[
k_{11}=\dfrac{\dsum\limits_{i=1}^{n}\lambda _{i}^{\ast }\beta _{i}^{1\ast }}{a}\cdot
\]%
$ii)$\textit{\ }%
\[
k_{21}=\dfrac{\Delta _{1}^{2}-k_{11}\Delta _{1}^{1}}{ak_{12}}\;.
\]%
$iii)$
\[
\beta _{i}^{2}=\dfrac{\lambda _{i}^{\ast }\beta _{i}^{1\ast }-k_{11}\beta_{i}^{1\ast }}{k_{21}}\;,\;\;\forall i\in \left\{ 1,2,\ldots ,n\right\}.
\]
\end{proposition}
\begin{proof}\textcolor[rgb]{1.00,1.00,1.00}{aaa}\vspace{0.2cm}\\
$i)$ Knowing that :
\[
\Delta _{1}^{1}=\dsum\limits_{i=1}^{n}\lambda _{i}^{\ast }\beta _{i}^{1\ast},
\]%
relationship(2.4) gives directly the result.\vspace{0.2cm}\\
$ii)$ The relationship (2.5) used for $l=2,\ j=1$ gives:%
\[
\Delta _{1}^{2}=k_{11}\Delta _{1}^{1}+k_{21}\Delta _{2}^{1},
\]%
and
\[
k_{21}=\dfrac{\Delta _{1}^{2}-k_{11}\Delta _{1}^{1}}{\Delta _{2}^{1}}=\dfrac{\Delta _{1}^{2}-k_{11}\Delta _{1}^{1}}{ak_{12}}\cdot
\]%
\vspace{0.1cm}\newline
$iii)$ Recall that the relationship (2.2) is given, after minimization of $J$, by:
\[
\lambda _{i}^{\ast }.\beta _{i}^{1\ast }=k_{11}\beta _{i}^{1\ast}+k_{21}.\beta _{i}^{2}\;,\;\;\forall i\in \left\{ 1,2,\ldots ,n\right\},
\]%
with a simple manipulation , we get :
\[
\beta _{i}^{2}=\dfrac{\lambda _{i}^{\ast }\beta _{i}^{1\ast }-k_{11}.\beta_{i}^{1\ast }}{k_{21}}\;,\;\;\forall i\in \left\{ 1,2,\ldots ,n\right\}.
\]
\end{proof}
\begin{remark}\label{rem}
This result says that the second column of the matrix $B$ is identified.
\end{remark}
\begin{proposition}\label{prop2} If the catenary system (S1) is
reversible, then the following relationships hold:\vspace{0.2cm}\\
$i)$
\[
k_{jj}=\dfrac{\Delta _{j}^{j}-k_{(j-1)j}\Delta _{j-1}^{j-1}}{\Delta_{j}^{j-1}},\ \forall \ j\in \left\{ 2,3,\ldots ,n-1\right\}.
\]%
$ii)$\textit{\ }%
\[
k_{(j+1)j}=\dfrac{\Delta _{j}^{j+1}-k_{(j-1)j}\Delta_{j-1}^{j}-k_{jj}\Delta _{j}^{j}}{\Delta _{j+1}^{j}},\ \forall \ j\in \left\{ 2,3,\ldots ,n-1\right\}.
\]%
$iii)$\textit{\ }%
\[
\beta _{i}^{j+1}=\dfrac{\lambda _{i}^{\ast }\beta _{i}^{j}-k_{(j-1)j}\beta_{i}^{j-1}-k_{jj}\beta _{i}^{j}}{k_{(j+1)j}},\ \forall i\in \left\{1,2,\ldots ,n\right\}.
\]
\end{proposition}
\begin{proof}
We recall that:
\[
\Delta _{p+1}^{p}\neq 0,\  \forall p\in \left\{ 1,2,\ldots ,n-1\right\}.
\]%
\vspace{0.15cm}\newline
Let $j$ be an integer in $\left\{ 2,3,\ldots,n-1\right\} $.\vspace{0.15cm}%
\newline
$i)$ The relationship (2.5) with $l=j$, gives:
\[
\Delta _{j}^{j}=k_{(j-1)j}\Delta _{j-1}^{j-1}+k_{jj}\Delta_{j}^{j-1}+k_{(j+1)j}\Delta _{j+1}^{j-1}.
\]%
Lemma \ref{lem2} indicates that:
\[
\Delta _{j+1}^{j-1}=0,
\]%
so
\[
k_{jj}=\dfrac{\Delta _{j}^{j}-k_{(j-1)j}\Delta _{j-1}^{j-1}}{\Delta_{j}^{j-1}}\cdot\vspace{0.25cm}
\]%
$ii)$ Relationship (2.5) when $l=j+1,$ gives:
\[
\Delta _{j}^{j+1}=k_{(j-1)j}\Delta _{j-1}^{j}+k_{jj}\Delta_{j}^{j}+k_{(j+1)j}\Delta _{j+1}^{j}.
\]%
By a simple manipulation of this equality, we obtain:
\[
k_{(j+1)j}=\dfrac{\Delta _{j}^{j+1}-k_{(j-1)j}\Delta_{j-1}^{j}-k_{jj}\Delta _{j}^{j}}{\Delta _{j+1}^{j}}\cdot\vspace{0.25cm%
}
\]%
$iii)$ The relationship (2.2) for all integer $i\in \left\{1,2,\ldots, n\right\} $ being:
\[
\lambda _{i}^{\ast }\beta _{i}^{j}=k_{(j-1)j}.\beta _{i}^{j-1}+k_{jj}\beta_{i}^{j}+k_{(j+1)j}\beta _{i}^{j+1},
\]%
with a simple manipulation of this equality we get:
\[
\beta _{i}^{j+1}=\dfrac{\lambda _{i}^{\ast }\beta _{i}^{j}-k_{(j-1)j}\beta_{i}^{j-1}-k_{jj}\beta _{i}^{j}}{k_{(j+1)j}}\;,\;\;\forall i\in \left\{
1,2,\ldots, n\right\}.
\]%
$k_{(j+1)j}\neq 0$ due to the reversibility of the system.
\end{proof}
\subsection{Identification theorem}\textcolor[rgb]{1.00,1.00,1.00}{aaa}\vspace{0.2cm}\\
Generally in compartmental analysis the identifiability of a $\ n-$compartmental
system requires the measurement of $(n-1)$ compartments or $(n-2)$
compartments\ under some conditions (see [8]). In catenary systems, we obtain the following
optimal result :

\begin{theorem}\label{th1}
Let (S1) be a reversible catenary
compartmental system. If the primary compartment is measurable then (S1) is identifiable.
\end{theorem}
\begin{proof}
This theorem is a consequence of Proposition \ref{prop1} and Proposition \ref{prop2}. In fact, owing to Proposition \ref{prop1}, we identify $k_{11},$ $k_{12}$,\ $k_{21}$%
, $\beta _{i}^{2},$ $\forall i\in \left\{ 1,2,\ldots ,n\right\} $ in this order as follow:

- firstly:%
\[
k_{11}=\dfrac{\dsum\limits_{i=1}^{n}\lambda _{i}^{\ast }\beta _{i}^{1\ast }}{a}\cdot
\]%
Since,
\[
k_{12}=-(k_{1e}+k_{11}),\vspace{0.1cm}
\]%
and knowing that (see [10]\ \&\ [11]):
\[
k_{1e}=\dfrac{-a}{\dsum\limits_{i=1}^{n}\dfrac{\beta _{i}^{1\ast }}{\lambda_{i}^{\ast }}},
\]%
we can get the value of $k_{12}$. Consequently, $k_{21}$ can be determined:
\[
k_{21}=\dfrac{\Delta _{1}^{2}-k_{11}\Delta _{1}^{1}}{ak_{12}},
\]%
as $\Delta _{1}^{2},\Delta _{1}^{1}\ $are known because the first column of
$B$ is known.\bigskip

- secondly:
\[
\beta _{i}^{2}=\dfrac{\lambda _{i}^{\ast }\beta _{i}^{1\ast }-k_{11}\beta_{i}^{1\ast }}{k_{21}},\ \forall i\in \left\{ 1,2,\ldots ,n\right\}.
\]%
This means that the first column \ thus the first row of the matrix $A$ are
identified and finally the second colomn of the matrix $B$ is
identified.\bigskip \newline
We suppose that at the step $j,\ j\in $ $\left\{ 1,2,\ldots ,n-1\right\} $,
the coefficients $k_{pp},$ $k_{p(p+1)},$ $k_{(p+1)p},$ $\beta _{i}^{(p+1)}$ $%
\forall i\in \left\{ 1,2,\ldots, n\right\} ,$ $\forall p\in \left\{
1,2,\ldots ,j-1\right\} $ had been identified. Owing to Proposition \ref{prop2}, we can identify $k_{jj},$ $k_{j(j+1)},$ $k_{(j+1)j},$ $\beta
_{i}^{j+1},$ $\forall i\in \left\{ 1,2,\ldots ,n\right\} $ in this order as follows:%
\[
k_{jj}=\dfrac{\Delta _{j}^{j}-k_{(j-1)j}\Delta _{j-1}^{j-1}}{\Delta_{j}^{j-1}},
\]%
and $\Delta _{j}^{j},\ \Delta _{j-1}^{j-1},\ \Delta
_{j}^{j-1}\ $are known because the columns $\left( j-1\right) $ and $j$ of
the matrix $B$ had been identified above (recurrence hypothesis). Then we can identify $k_{j(j+1)}$ because:
\[
k_{j(j+1)}=-(k_{j(j-1)}+k_{jj}),
\]%
and $k_{j(j-1)}\ $is known by recurrence hypothesis. Consequently, $%
k_{(j+1)j}$ can be determined since:%
\[
k_{(j+1)j}=\dfrac{\Delta _{j}^{j+1}-k_{(j-1)j}\Delta_{j-1}^{j}-k_{jj}\Delta _{j}^{j}}{\Delta _{j+1}^{j}}=\dfrac{\Delta_{j}^{j+1}-k_{(j-1)j}\Delta _{j-1}^{j}-k_{jj}\Delta _{j}^{j}}{k_{12}k_{23}\cdots ak_{(j-1)j}},
\]%
and$\ \Delta _{j}^{j+1},\ \Delta _{j-1}^{j}\ ,\ \Delta _{j}^{j}\ $are known because the columns $\left( j-1\right) $ and $j$ of $B$ had been identified above (recurrence hypothesis).\\
We have :%
\[
\beta _{i}^{j+1}=\dfrac{\lambda _{i}^{\ast }\beta _{i}^{j}-k_{(j-1)j}\beta_{i}^{j-1}-k_{jj}\beta _{i}^{j}}{k_{(j+1)j}}\;,\;\;\forall i\in \left\{1,2,\ldots ,n\right\},
\]%
and all the parameters of $\beta _{i}^{j+1}$ expression are known, so it can be identified. This means that the column $j$\ then the row $j$ of the matrix $A$ are identified and finally the column \ $j+1$\ of the matrix $B$
is identified for all $j\in \left\{ 1,2,\ldots ,n-1\right\} $.\\
To identify $k_{nn}$, it suffices to remark that:%
\[
k_{nn}=-k_{n(n-1)},
\]%
and the coefficient $k_{n(n-1)}$ had been identified above.
\end{proof}
\subsection{Identification's algorithm.}\textcolor[rgb]{1.00,1.00,1.00}{aaa}\vspace{0.2cm}\\
The results of the Proposition \ref{prop1} can be used for the algorithm's
initialization and the results of the Proposition \ref{prop2} can be used
to define iterations which permits to identify at each step the
corresponding column and row of the tridiagonal compartmental matrix and the
corresponding column of the partial measures matrix.
\section{Conclusion}
The problem of the minimal number of compartment to measure for identify a
compartmental system had been totaly solved for reversibles catenarys
systems. Moreover, with the obtained results, we can construct an
identification algorithm for reversibles catenarys systems with just the
measure of the first compartment.

\end{document}